\documentclass[12pt,oneside,reqno]{amsart}
\usepackage{mathrsfs, amsmath,amssymb}
\usepackage{fullpage}
\numberwithin{equation}{section}
\usepackage{graphicx}
\usepackage{svg}
\newtheorem{teo}{Theorem}[section]
\newtheorem{pro}[teo]{Proposition}
\newtheorem{lem}[teo]{Lemma}
\newtheorem{cor}[teo]{Corollary}
\newtheorem{rem}[teo]{Remark}

\title{Free mutual information for two projections}

\author[T. Hamdi]{Tarek Hamdi}
\address{Department of Management Information Systems \\ College of Business Administration \\ Al-Rass, Al-Qassim University \\ Saudi Arabia 
and Laboratoire d'Analyse Math\'ematiques et applications \\ LR11ES11 \\ Universit\'e de Tunis El-Manar \\ Tunisie
}
\email{tarek.hamdi@mail.com}

\begin{document}
\begin{abstract}
The present paper provides a proof of $i^*( \mathbb{C}P+\mathbb{C}(I-P); \mathbb{C}Q+\mathbb{C}(I-Q) )=-\chi_{orb}(P,Q)$ for two projections $P,Q$ without any extra assumptions. An analytic approach is adopted to the proof, based on a subordination result for the liberation process of symmetries associated with $P,Q$.
\end{abstract}
\maketitle

\section{Introduction}
 In classical information theory, the mutual information $I(X,Y)$ between two  random variables $X,Y$  can be formally expressed in terms of their Shannon-Gibbs entropies as follows
 \begin{equation*}
 I(X,Y)=H(X)+H(Y)-H(X,Y).
\end{equation*}
Motivated by the above expression, Voiculescu introduced the so-called free mutual information $i^*$ in \cite{Voi}, via the liberation theory in free probability, as a way  to have invariants to distinguish von Neumann algebras. In this way, Hiai, Miyamoto and Ueda  introduced  in  \cite{Hia-Miy-Ued}, \cite{Ued} the so-called orbital free entropy $\chi_{orb}$  which also plays a role of free analog of mutual  information (see also its new approaches due to Biane and Dabrowski in \cite{Bia-Dab}). 
This quantity  $\chi_{orb}$ has many properties in common with $i^*$, but there is no general relationship between them so far. A few years ago, Collins and Kemp \cite{Col-Kem} gave a proof of the identity 
\begin{equation*}
 i^*\left( \mathbb{C}P+\mathbb{C}(I-P); \mathbb{C}Q+\mathbb{C}(I-Q) \right)=\chi_{orb}(P)+\chi_{orb}(Q)-\chi_{orb}(P,Q)=-\chi_{orb}(P,Q)
\end{equation*} 
for two projections of traces $\frac{1}{2}$ and the same result was subsequently proved by Izumi and Ueda in \cite{Izu-Ued} with a completely independent proof. 
Motivated both by the ideas in \cite{Izu-Ued} and the heuristic argument in \cite[Section 3.2]{Hia-Ued}, 
we developed  in \cite{Tar} a theory of subordination for the liberation process of symmetries which allowed us to establish some partial results generalizing the equality  $i^*=-\chi_{orb}$ for two projections.

Throughout the present notes, let $(\mathscr{A},\tau)$ be a $W^*$-probability space and $U_t, t\in[0,\infty)$ a free unitary Brownian motion in $(\mathscr{A},\tau)$ with $U_0={\bf1}$. 
For given two projections $P,Q$ in $\mathscr{A}$ that are independent of $(U_t)_{t\geq0}$ we denote by $R=2P-{\bf1}$ and $S=2Q-{\bf1}$ the associated symmetries.
Let $a=|\tau(R)-\tau(S)|/2$ and $b=|\tau(R)+\tau(S)|/2$.
In a recent paper (\cite{Tar}), we studied the spectral distribution $\nu_t$ of the unitary operator $RU_tSU_t^*$ for arbitrary value of $a,b$.
The key result proved in \cite{Tar} is the following characteristic equation:
\begin{equation}\label{subor}
H(t,\phi_t(z))^2-H(\infty,\phi_t(z))^2=H(0,z)^2-H(\infty,z)^2
\end{equation}
for given initial data $H(0,z)$ and $t>0$, where $\phi_t$ is a flow  defined on a region $\Omega_t$ of $\mathbb{D}$, the function
\begin{equation*}
H(t,z):=\int_{\mathbb{T}}\frac{\xi+z}{\xi-z}d\nu_t(\xi)
\end{equation*}
is  the  Herglotz transform of the measure $\nu_t$ and 
\begin{equation*}
H(\infty,z)=\sqrt{1+4 z\frac{\alpha \beta \left(1+z\right)^2+ \left(\alpha-\beta \right)^2z}{\left(1-z^2\right)^2}}.
\end{equation*}
Note that this type of result was proved in \cite{Dem-Hmi} in the special case $P=Q$.
In fact, the equality \eqref{subor} turns out to an exact subordination relation (see \cite[Proposition 3.4]{Tar}): 
$K(t,z)=K(0,\eta_t(z))$ for a subordinate function $\eta_t$ (the inverse of $\phi_t$) which extend to a homeomorphism from the closed unit disc $\overline{\mathbb{D}}$ to $\overline{\Omega_t}$ where
 \begin{equation*}
K(t,z):=\sqrt{H(t,z)^2-\left(a\frac{1-z}{1+z}+b\frac{1+z}{1-z} \right)^2}.
\end{equation*}
This subordination relation is used in  \cite[Lemma 4.4]{Izu-Ued} to prove that the equality $i^*=-\chi_{orb}$ holds, for any two projections, under the assumption that $K(t,.)$ define a function of Hardy class $H^{3}(\mathbb{D})$ for any $t>0$. 
Note that the function $H$ there is exactly  our $\frac{1}{4}K^2$ (notation and definitions that are used throughout this paper are the same as in \cite{Tar}). Thus we mainly investigate the boundary behavior of the function $K(t,z)$ in what follows.  Our main result here is the following.
\begin{teo}\label{main}
The equality
$
i^*\left( \mathbb{C}P+\mathbb{C}(I-P); \mathbb{C}Q+\mathbb{C}(I-Q) \right)=-\chi_{orb}\left(P,Q\right)
$
holds for any pair of projections.
\end{teo}
The paper has four sections including this introduction. Section 2 contains remainder of  the main results proved in \cite{Tar} and preliminaries about boundary results associated with subordinate functions.
Section 3 deals with regularity properties of the regions $\Omega_t$. Section 4 gives a proof of the Theorem \ref{main}.

\section{Remainder and preliminaries}

We studied in \cite{Tar} the relationship between the spectral distributions $\mu_t$ and $\nu_t$   of the self adjoint operator $PU_tQU_t^*$ and the unitary operator $RU_tSU_t^*$  where the projections $\{P,Q\}$ and the  symmetries $\{R,S\}$ are associated in $ \mathscr{A}$ and freely independent from $U_t$. 
Let $\psi(t,z)$ be the moments generating function of the spectral measure $\mu_t$ and set $H_{\mu_t}(z)=1+2\psi(t,z)$. From \cite[Corollary 4.2]{Tar}, the  Herglotz transform of the measure $\nu_t$ satisfies
\begin{equation}\label{relationship}
 H(t,z)=\frac{z-1}{z+1}H_{\mu_t}\left(\frac{4z}{(1+z)^2}\right)-2(\alpha+\beta)\frac{z}{z^2-1}
\end{equation}
 where $\alpha=\tau(R)$ and $\beta=\tau(S)$. Thus, according to \cite[Theorem 1.4]{Col-Kem}, $H(t,z)$ is analytic in both $z\in\mathbb{D}$ and $t>0$. Moreover, from  \cite[Proposition 2.3]{Tar}, $H(t,z)$ solves the pde
 \begin{equation}\label{pde}
\partial_tH+\frac{z}{2}\partial_zH^2= \frac{2 z \left(\alpha z^2+2 \beta z+\alpha\right) \left(\beta z^2+2  \alpha z+\beta\right)}{\left(1-z^2\right)^3}.
\end{equation}
 Using the method of characteristic, we deduce the existence of a subordinator function $\phi_t$ satisfying
 the following coupled ordinary differential equations (ode for short)
\begin{equation}\label{Low1}
\partial_t\phi_{t} = \phi_{t} H(t,\phi_{t}), \quad \phi_{0}(z) = z,
\end{equation}
\begin{equation}\label{Low2}
\partial_t \left[H(t,\phi_{t})\right] =\frac{4(\alpha^2+\beta^2)\phi_t^2(1+\phi_t^2)+2\alpha\beta \phi_t(1+6\phi_t^2+\phi_t^4)}{(1-\phi_t^2)^{3}}.
\end{equation}
Recall that the ode \eqref{Low1} is nothing else but the L\"owner equation determined by the 1-parameter family of probability measures $t\mapsto \nu_t$. Then  (see  \cite[Theorem 4.14]{Law}) it defines a unique 1-parameter family of conformal transformations $\phi_t$ from $\Omega_{t}:=\{z\in\mathbb{D}: t<T_z\}$ onto $\mathbb{D}$ with $\phi_t(0)=0$ and $\partial_z\phi_t(0)=e^t$, where $T_z$ is the supremum of all $t$ such that $\phi_t(z)\in\mathbb{D}$ for fixed $z\in\mathbb{D}$. 
Integrating the ode \eqref{Low1}, we get
\begin{align}\label{exp}
\phi_t(z)=z\exp\left(\int_0^tH(s,\phi_s(z))ds\right).
\end{align}

Its known that (see, e.g., \cite[Remark 4.15]{Law}) $\phi_t$ is invertible and the inverse $\eta_t:=\phi_t^{-1}$ from $\mathbb{D}$ onto $\Omega_t$ solves the radial L\"owner pde:
\begin{equation*}
\partial_t\eta_t(z)=-z\partial_z\eta_t(z)H(t,z), \quad \eta_0(z)=z.
\end{equation*}
 The function $\phi_t$ satisfies the properties in \cite[Theorem 4.4, Proposition 4.5]{Bel-Ber}. Thus, we have 
\begin{pro}\cite{Bel-Ber}\label{Bercovici}
\begin{enumerate}
\item $\eta_t$ extends continuously to the boundary of $\mathbb{D}$, and $\eta_t$ is one-to-one on $\overline{\mathbb{D}}$. 
\item $\Omega_{t}$ is a simply connected domain bounded by a simple closed curve. This domain coincide with $\eta_t( \mathbb{D})$ and its boundary is $\eta_t(\mathbb{T})$.
\item If $\zeta\in\mathbb{T}$ satisfies $\eta_t(\zeta)\in \mathbb{D}$,  $\eta_t$ can be continued analytically to a neighborhood of $\zeta$.
\item A point $\zeta\in\mathbb{T}$ belong to the boundary of $\Omega_{t}$, if and only if the limit $l_t=\phi_t(\zeta)=\lim_{r\rightarrow1^-}\phi_t(r\zeta)$ exists,  $l_t\in\mathbb{T}$ and $\frac{\zeta}{l_t}\partial_z\phi_t(\zeta)\in[0,1)$.
\item If $\zeta\in\mathbb{T}\cap \overline{\Omega_{t}}$ and $\frac{\zeta}{l_t}\partial_z\phi_t(\zeta)>0$ then $\phi_t(r\zeta)$ approaches $\phi_t(\zeta)\in\mathbb{T}$ nontangentially as $r\rightarrow1^-$.
\end{enumerate}
\end{pro}
Here is a sample application.
\begin{cor}\label{Phi}
The function
\begin{equation*}
\Phi_t(z):=a\frac{1-\phi_t(z)}{1+\phi_t(z)}+b\frac{1+\phi_t(z)}{1-\phi_t(z)}
\end{equation*}
has a continuous extension to $\mathbb{T}\cap \overline{\Omega_{t}}$.
\end{cor}
\begin{proof}
By Proposition \ref{Bercovici}, $\phi_t$ has a continuous extension to $\mathbb{T}\cap \overline{\Omega_{t}}$. Assume both $a\neq0$ and $b\neq0$, then according to \cite[Lemma 3.7]{Tar} the boundary $\partial{\Omega_t}=\eta_t(\mathbb{T})$ does not contain the points $\pm1$. More precisely, the boundary $\partial\Omega_t$ intersect the x-axis at two points $x(t)_\pm$ from either side of the origin with $(x(t)_-,x(t)_+)\subset(-1,1)$ and $\phi_t(x(t)_\pm)=\pm1$. 
Thus $\phi_t(z)$ does not take the values $\pm1$ on $\mathbb{T}\cap \overline{\Omega_{t}}$ and the result follows immediately.
\end{proof}

Finally  we notice that, by  \cite[Proposotion 3.5]{Tar}, the transformation $\phi_t$  coincides on the interval $\mathbb{R}\cap\Omega_t=(x(t)_-,x(t)_+)$ with
\begin{align*}
\frac{\left(\sqrt{\left(b^2-a^2-c-d e^{t\sqrt{c}}\right)^2-4a^2c}-\sqrt{\left(b^2-a^2+c-d e^{t\sqrt{c}}\right)^2-4b^2c}\right)^2}{4cde^{t\sqrt{c}}},
\end{align*}
where $c=c(z):=K(0,z)^2+ (a+b)^2$ and
\begin{align*}
d(z)=\frac{1}{x}\left[b^2x^2-\left(\sqrt{c}-\sqrt{c-(c-a^2+b^2)x+b^2x^2}\right)^2 \right],\quad x=\frac{-4z}{(1-z)^2}.
\end{align*}

\section{Regularity properties of $\Omega_t$}

Recall from Proposition \ref{Bercovici} that $\Omega_{t}$ is simply connected and its boundary is a simple closed curve.
We use here polar coordinates to provide explicit descriptions for $\Omega_{t}$  and its boundary.
 The following result shows that $(\Omega_t)_{t>0}$ is decreasing on $\mathbb{D}$.
\begin{lem} Given $0<s<t$, then $\Omega_t\subset\Omega_s\subset\mathbb{D}$.
\end{lem}
\begin{proof} 
Since $\Re H(u,\phi_u(z))>0$ for any $u>0$, we have
\begin{align*}
\int_0^s\Re H(u,\phi_u(z))du\leq \int_0^t\Re H(u,\phi_u(z))du
\end{align*}
for $s<t$.
Thus, $|\phi_s(z)|\leq |\phi_t(z)|$ and hence $\Omega_{t}\subset\Omega_{s}$.
\end{proof}
For fixed $\zeta\in\mathbb{T}$ and $r\in(0,1)$, define
\begin{align*}
h_t(r,\zeta)=1-\int_0^t\frac{1-|\phi_s(r\zeta)|^2}{-\ln r}\int_{\mathbb{T}}\frac{1}{|\xi-\phi_s(r\zeta)|^2}d\nu_s(\xi)ds.
\end{align*}
Then, we have 
\begin{align}\label{logmod}
\ln \left|\phi_t(r\zeta)\right|=\ln r+\Re\int_0^tH(s,\phi_s(r\zeta))ds=(\ln r) h_t(r,\zeta).
\end{align}
To study the boundary of $\Omega_t$, we need the following result.
\begin{lem} \label{limit}
Given $t>0$ and $e^{i\theta}\in\mathbb{T}\cap\overline{\Omega_{t}}$. Let $\phi_t'(z)=\partial_z\phi_t(z)$, then 
\begin{align*}
\lim_{r\rightarrow 1^-}\frac{1-|\phi_t(re^{i\theta})|^2}{-\ln r}=\frac{2e^{i\theta}}{\phi_t(e^{i\theta})}\phi_t'(e^{i\theta})\in[0,2).
\end{align*}
\end{lem}
\begin{proof} 
Let $\theta\in[-\pi,\pi]$. According to \cite[Proposition 4.5]{Bel-Ber}, the limit
\begin{align*}
\lim_{r\rightarrow 1^-}\frac{\phi_t(e^{i\theta})-\phi_t(re^{i\theta})}{(1-r)\phi_t(e^{i\theta})}=\frac{e^{i\theta}}{\phi_t(e^{i\theta})}\phi_t'(e^{i\theta})\in[-\infty,1)
\end{align*}
exists and is non-negative when  $e^{i\theta}\in\mathbb{T}\cap\overline{\Omega_{t}}$.
  Hence, keeping in mind that $|\phi_t(e^{i\theta})|=1$, the assertion follows by the following elementary calculus.
  \begin{align*}
\lim_{r\rightarrow 1^-}\frac{1-|\phi_t(re^{i\theta})|^2}{-\ln r}
&=\lim_{r\rightarrow 1^-}\frac{(1-r)|\phi_t(e^{i\theta})+\phi_t(re^{i\theta})|^2}{-\ln r}\frac{1-|\phi_t(re^{i\theta})|^2}{(1-r)|\phi_t(e^{i\theta})+\phi_t(re^{i\theta})|2}
\\&=\lim_{r\rightarrow 1^-}\frac{(1-r)|\phi_t(e^{i\theta})+\phi_t(re^{i\theta})|^2}{-\ln r}\Re\left[\frac{\phi_t(e^{i\theta})-\phi_t(re^{i\theta})}{(1-r)\phi_t(e^{i\theta})}\right]
\\&=2|\phi_t(e^{i\theta})|^2\frac{e^{i\theta}}{\phi_t(e^{i\theta})}\phi_t'(e^{i\theta}).
\end{align*}
\end{proof}
From Lemma \ref{limit}, we have for fixed $\theta\in[-\pi,\pi]$
 \begin{align*}
\lim_{r\rightarrow 1^-}h_t(r,e^{i\theta})=&\lim_{r\rightarrow 1^-}\frac{\ln\left|\phi_t (re^{i\theta})\right|}{\ln r}
\\&=\lim_{r\rightarrow 1^-}\frac{\ln\left|\phi_t (re^{i\theta})\right|}{1-|\phi_t(re^{i\theta})|^2}\frac{1-|\phi_t(re^{i\theta})|^2}{\ln r}
\\&=\frac{e^{i\theta}}{\phi_t(e^{i\theta})}\phi_t'(e^{i\theta})\in[-\infty,1).
\end{align*}
Define $R_t:[-\pi,\pi]\rightarrow [0,1]$ and  $h_t:[-\pi,\pi]\rightarrow\mathbb{R}\cup\{-\infty\}$ as follows:
\begin{align*}
R_t(\theta)&=\sup\left\{r\in(0,1):h_t(r,e^{i\theta})>0\right\}
\\&=\sup\left\{r\in(0,1):\int_0^t\frac{1-|\phi_s(re^{i\theta})|^2}{-\ln r}\int_{\mathbb{T}}\frac{1}{|\zeta-\phi_s(re^{i\theta})|^2}d\nu_s(\zeta)ds<1\right\},
\end{align*}
and $h_t(\theta)=\lim_{r\rightarrow 1^-}h_t(r,e^{i\theta})$.
Let
\begin{align*}
I_{t}&=\left\{\theta\in[-\pi,\pi]:h_t(\theta)<0\right\}
=\left\{\theta\in[-\pi,\pi]:\frac{e^{i\theta}}{\phi_t(e^{i\theta})}\phi_t'(e^{i\theta})<0\right\}.
\end{align*}
The next result gives a description of $\Omega_{t}$  and its boundary. 
\begin{pro}
For any $t>0$, we have
\begin{enumerate}
\item $\Omega_{t}=\{re^{i\theta}: h_t(r,e^{i\theta})>0\}$ 
\item $\partial\Omega_{t}\cap\mathbb{D}=\{re^{i\theta}: h_t(r,e^{i\theta})=0\ {\rm and}\ \theta\in I_{t}\}$.
\item $\partial\Omega_{t}\cap\mathbb{T}=\{e^{i\theta}: h_t(r,e^{i\theta})=0\ {\rm and}\ \theta\in I_{t}^c\}$.
\end{enumerate}
\end{pro}
\begin{proof}
From  \eqref{logmod}, we have 
\begin{align*}
re^{i\theta}\in\Omega_t \Leftrightarrow |\phi_t(re^{i\theta})|<1\Leftrightarrow h_t(r,e^{i\theta})>0 
\end{align*} 
which proves (1). Referring also to \eqref{logmod} and by Proposition \ref{Bercovici}, we have
\begin{align*}
re^{i\theta}\in\partial\Omega_{t} \Leftrightarrow |\phi_t(re^{i\theta})|=1  \Leftrightarrow  h_t(r,e^{i\theta})=0.
\end{align*}
Hence the desired assertions follow since we have by definition of $I_t$: $\theta\in I_{t}$ if and only if $R_t(\theta)<1$.
\end{proof}
\begin{rem}
Note that this description is analogous to the one obtained by Zhong in \cite[Theorem 3.2]{Zho} when $a=b=0$ (i.e. $\tau(P)=\tau(Q)=1/2$) where the measure $\nu_{t/2}$  becomes identical to the probability distribution of a free unitary Brownian motion with initial distribution $\nu_0$ (see \cite[Remark 4.7]{Tar}).
\end{rem}

\section{Proof of the main result}

Our approach to Theorem \ref{main} relies on a study of the boundary behavior of $K(t,.)$  for any $t>0$.
Form the identity $K(t,z)=K(0,\eta_t(z))$ together with the equality $\eta_t(\overline{\mathbb{D}})=\overline{\Omega_t} $, it suffices  to investigate the behavior of $K(0,.)$ on the boundary 
 $\partial{\Omega_{t}}\cap\mathbb{T}$ since  $K(0,.)$ is analytic in $\mathbb{D}$, and then it extends analytically to $\overline{\Omega_{t}}\cap\mathbb{D}$ for any $t>0$. Without loss of generality, we may restrict our study to a subset $V_t$  of $\Omega_t$ which does not meet the boundary $\partial{\Omega_{t}}\cap\mathbb{D}$ and  whose boundary in  $\mathbb{T}$ is exactly $\partial{\Omega_{t}}$. 
Recall that  the function $\phi_t$ is analytic in $\Omega_t$ and has a continuous extension to $\overline{\Omega_{t}}$ by Proposition \ref{Bercovici}.
The identity $\partial_z\phi_t(0)=e^t$ and  \eqref{exp} imply that
\begin{align}\label{flow1}
\frac{\phi_t(z)}{z}=\exp\left(\int_0^tH(s,\phi_s(z))ds\right)
\end{align}
extends to $z\in\overline{\Omega_t}$. Since
\begin{align*}
|\Im H(s,\phi_s(z))|&=\left| \int_{\mathbb{T}}\frac{2\Im(\overline{\xi}\phi_s(z))}{|\xi-\phi_s(z)|^2}d\nu_s(\xi) \right|
\\\leq& \int_{\mathbb{T}}\frac{2\left|\phi_s(z) \right|}{|\xi-\phi_s(z)|^2}d\nu_s(\xi)
\\=&\frac{2|\phi_s(z)|}{1-|\phi_s(z)|^2}\Re H(s,\phi_s(z)),
\end{align*}
we then have that
\begin{align*}
\left|\Im \int_0^tH(s,\phi_s(z))ds\right|\leq&\int_0^t\frac{2|\phi_s(z)|}{1-|\phi_s(z)|^2}\Re H(s,\phi_s(z))ds
\\\leq&\frac{2|\phi_t(z)|}{1-|\phi_t(z)|^2}\int_0^t\Re H(s,\phi_s(z))ds
\\\leq&\frac{-2|\phi_t(z)|\ln|z|}{1-|\phi_t(z)|^2}.
\end{align*}
where the second inequality follows from the fact that $|\phi_s(z)|\leq|\phi_t(z)|$ for $s\leq t$ and the last inequality is due to the definition of $R_t(\theta)$.
Let
\begin{equation*}
h_t(z):=\frac{-2|\phi_t(z)|\ln|z|}{1-|\phi_t(z)|^2},
\end{equation*}
then from Lemma \ref{limit}, $h_t(.)$ is bounded on $V_t$ for $t\geq0$. In particular, for $t=0$ we have
\begin{align*}
|h_0(z)|=\frac{-2|z|\ln|z|}{1-|z|^2}\leq1.
\end{align*}
Therefore there is some $t_0>0$ such that $|h_t(z)|<\pi$, for $t\in[0,t_0)$ 
 and thus, we can take logarithms in both sides of \eqref{flow1}. This implies that the exponent $\int_0^tH(s,\phi_s(z))ds$ in the right-hand side of \eqref{flow1} has a continuous extension to $\overline{V_t}\cap\mathbb{T}$ for $t\in(0,t_0)$.
Next, we can  apply integration by parts to write
\begin{align}\label{iip}
tH(t,\phi_t(z))=\int_0^tH(s,\phi_s(z))ds +\int_0^t s\partial_s[H(s,\phi_s(z))]ds,
\end{align}
where we recall from   \eqref{Low2} that
\begin{align*}
s\partial_s[H(s,\phi_s(z))]= 2s\phi_s(z)\left(b^2\frac{1+\phi_s(z)}{(1-\phi_s(z))^3}-a^2\frac{1-\phi_s(z)}{(1+\phi_s(z))^3}\right).
\end{align*}
The function $\phi_s(z)$ is  continuous jointly in both variables for $s\in[0,t]$ and $z\in\overline{V_t}$ since we have $\overline{V_t}\subset\overline{\Omega_t}\subset\overline{\Omega_s}$, hence so does $s\partial_s[H(s,\phi_s(z))]$ too since $\phi_s(z)\neq\pm1$. Thus, $\int_0^t s\partial_s[H(s,\phi_s(z))]ds$ is a continuous function of $z$ on $ \overline{V_t}$
as consequence of the theorem of continuity under integral sign. It follows from \eqref{iip} that $H(t,\phi_t(z))$ extends continuously to $\overline{V_t}\cap\mathbb{T}$.
Now, the identity $K(t,\phi_t(z))=K(0,z)$, implies that $H(t,\phi_t(z))$ rewrites as $\sqrt{K(0,z)^2+\Phi_t(z)^2}$ with
\begin{align*}
\Phi_t(z)=a\frac{1-\phi_t(z)}{1+\phi_t(z)}+b\frac{1+\phi_t(z)}{1-\phi_t(z)}.
\end{align*}
Since   $ \Phi_t(z)$ has a continuous extension to the boundary $\mathbb{T}\cap\overline{V_t}$ by Corollary \ref{Phi}, we deduce that $K(0,.)$ extends continuously to $\mathbb{T}\cap\overline{V_t}$ for every $t\in(0,t_0)$.
\begin{rem}
We notice that $\overline{V_t}\cap\mathbb{T}=\overline{\Omega_t}\cap\mathbb{T}$. Since $(\Omega_t)_{t>0}$  is decreasing on $\mathbb{D}$, $K(0,.)$ extends continuously to $\overline{\Omega_t}\cap\mathbb{T}$ for every $t\geq t_0$. 
\end{rem} 
The discussions so far are summarized as follows.

\begin{pro}\label{extension}
For every $t>0$, the function $K(0,.)$ extends analytically to $\partial{\Omega_{t}}\cap\mathbb{D}$ and has a continuous extension to $\partial{\Omega_{t}}\cap\mathbb{T}$.
\end{pro}
This bring us to the proof of the main result.

\begin{proof}[Proof of Theorem \ref{main}]
Form the identity $K(t,z)=K(0,\eta_t(z))$ and the fact that $\eta_t(\overline{\mathbb{D}})=\overline{\Omega_t} $ together with Proposition \ref{extension}, we deduce that the function $K(t,.)$ has a continuous extension to $\overline{\mathbb{D}}$, for any $t>0$. Thus, $K(t,.)$ becomes  a function of Hardy class $H^{\infty}(\mathbb{D})$  for every $t>0$ and the desired result follows easily from \cite[ Lemma 5.1]{Tar}.
\end{proof}

\end{document}